\documentclass[a4paper,12pt]{amsart}

\usepackage{latexsym}
\usepackage{amssymb,amsmath}
\usepackage{graphics}
\usepackage{url}
\usepackage[vcentermath]{youngtab}
  
\usepackage{booktabs}  
\usepackage{caption}
\usepackage{tikz}
\usetikzlibrary{snakes}
\usetikzlibrary{arrows}
\tikzstyle{empty}=[circle,draw=black!80,thick]
\tikzstyle{emptyn}=[circle,draw=black!80,fill=white,scale=0.5] 
\tikzstyle{nero}=[circle,draw=black!80,fill=black!80,thick]

\newtheorem{teo}{Theorem}[section]
\newtheorem{corollary}[teo]{Corollary}

\newtheorem{proposition}[teo]{Proposition}

\newtheorem{lemma}[teo]{Lemma}

\theoremstyle{definition}
\newtheorem{definition}[teo]{Definition}

\newtheorem{example}[teo]{Example}

\renewcommand{\epsilon}{\varepsilon}

\newcommand{\fS}{{\mathfrak{S}}}
\newcommand{\Ind}{\mathrm{Ind}}
\newcommand{\res}{\big\downarrow}
\newcommand{\ind}{\big\uparrow}
\newcommand{\Res}{\mathrm{Res}}
\newcounter{thmlistcnt}
	{\setcounter{thmlistcnt}{0}%
	\begin{list}{\emph{(\roman{thmlistcnt})}}{%
		\usecounter{thmlistcnt}%
		\setlength{\topsep}{0pt}%
		\setlength{\leftmargin}{27pt}%
		\setlength{\itemsep}{0pt}%
		\setlength{\labelwidth}{20pt}
		\setlength{\itemindent}{0pt}}%
	}%
	{\end{list}}%

\title[Characters of Odd Degree of Symmetric Groups]{Characters of Odd Degree of Symmetric Groups}
\author{Eugenio Giannelli}
\address[E.~Giannelli]{Department of Mathematics, University of Kaiserslautern,
P.O. Box 3049, 67655 Kaiserslautern, Germany}
\email{gianelli@mathematik.uni-kl.de}

\begin{document}

\thanks{The author gratefully acknowledges financial support by the
  ERC Advanced Grant 291512.}

\begin{abstract}
We construct a natural bijection between odd-degree irreducible characters of $\fS_n$ and linear characters of its Sylow $2$-subgroup $P_n$. When $n$ is a power of $2$, we show that such a bijection is nicely induced by the restriction functor. 
We conclude with a characterization of the irreducible characters $\chi$ of $\fS_n$ such that the restriction of $\chi$ to $P_n$ has a unique linear constituent. 
\end{abstract}

\maketitle
\thispagestyle{empty}

\section{Introduction}   \label{sec:intro}

J. McKay conjectured in his landmark paper
\cite{McK} that the number $n_2(G)$ of odd-degree irreducible
characters of a finite group $G$ equals $n_2({\bf N}_G(P))$,
where $P$ is a Sylow 2-subgroup of $G$.
This is now known as the $p=2$ case of the
{\sl McKay Conjecture} and has been recently proved after
more than 40 years in \cite{MalleSpath}.

Sometimes, 
we do not only have that 
$n_p(G)=n_p({\bf N}_G(P))$, but it is also possible to determine a natural bijection between
${\rm Irr}_{p'}(G)$ and ${\rm Irr}_{p'}({\bf N}_G(P))$, where
${\rm Irr}_{p'}(G)$ is the set of complex irreducible characters
of $G$ of degree not divisible by $p$.
This occurs for solvable groups and $p=2$ (see \cite{I}),
for groups with a normal $p$-complement, using the Glauberman correspondence and for groups having self-normalising Sylow $p$-subgroup
for $p$ odd (see \cite{N} and \cite{NTiepVallejo}).

In this article we fix $p=2$ and we focus our attention on the symmetric groups. It is well known that in this case, if $P$ is a Sylow $2$-subgroup of the symmetric group $\fS_n$ then $N_{\fS_n}(P)=P$. 
Our main result is the following theorem.

\medskip

\begin{teo}\label{thmA}
Let $n\in\mathbb{N}$ and let $P$ be a Sylow
$2$-subgroup of the symmetric group $\fS_n$. 
\begin{enumerate}
\item[(i)] If $\chi \in {\rm Irr}(\fS_n)$  has odd degree,
then $\chi$ canonically determines a linear character $\chi^* \in {\rm Irr}(P)$, and
the map $\chi \mapsto \chi^*$ defines a natural bijection between
the odd-degree irreducible characters of $\fS_n$ and the linear characters of $P$.

\item[(ii)] Moreover, if $n$ is a power of $2$, then $\chi^*$ is the only irreducible
constituent of the restriction of $\chi$ to $P$ which has odd degree.
\end{enumerate}
\end{teo}

In general, $\chi^*$ is not a constituent of the restriction of $\chi$ to $P$, although this should
not constitute a surprise (as the same already happens in the key case where
$G$ has a normal $p$-complement). 

McKay conjecture for symmetric groups was first proved with a beautiful counting argument 
by J. Olsson in \cite{Olsson}.
Later, bijections between ${\rm Irr}_{p'}(\fS_n)$ and ${\rm Irr}_{p'}({\bf N}_{\fS_n}(P))$ were given by P. Fong in \cite{Fong}, when checking  the Isaacs-Navarro conjecture (first stated in \cite{IN1}). These bijections are not choice-free (canonical), in fact their definition depends on some arbitrarily chosen labelling of the characters of $P$ (see \cite[Section 2]{Fong} for precise details).
\medskip

The second part of Theorem \ref{thmA} has a particular story.
G.~Navarro brought to our attention that 
in \cite{IN2} it is explained that J.~Alperin privately communicated to M. Isaacs and G.~Navarro that if $\chi$ is an irreducible odd-degree character of $\fS_{2^n}$, then $\chi\res_{P}$ has a unique linear constituent. 
We could not find a proof of this fact anywhere in the literature. With the kind permission of J.~Alperin 
we take the opportunity to present here our own proof of this correspondence (Theorem \ref{Thm:Alperin} below). 

\medskip
 
We would like to remark that, building on the bijection described in part (i) of Theorem  \ref{thmA}, it is possible to construct an explicit canonical bijection between $\mathrm{Irr}_{2'}(\mathfrak{S}_n)$ and $\mathrm{Irr}_{2'}(M)$, where $M$ is any maximal subgroup of $\mathfrak{S}_n$ of odd index. This is done in \cite[Theorem D]{GKNT}.  
Moreover, the result obtained in part (ii) of Theorem \ref{thmA} is fundamental to determine a natural correspondence between $\mathrm{Irr}_{2'}(G)$ and $\mathrm{Irr}_{2'}(N_G(P))$, where $G=GL_n(q)$, $q$ is an odd prime power and $P$  
is a Sylow $2$-subgroup of $G$. This is done in \cite[Theorem E]{GKNT}.  
 
\medskip
 
 In Section \ref{Sec:lastt}, we study the number of linear constituents in the restriction to a Sylow $2$-subgroup of any irreducible character of $\fS_n$. In particular, we prove the following result.
 
\begin{teo}\label{thmB}
Let $\chi\in\mathrm{Irr}(\fS_n)$, 
and let $P_n$ be a Sylow $2$-subgroup of $\fS_n$. 
\begin{enumerate}
\item[(i)] The restriction of $\chi$ to $P_n$ has a linear
constituent. 
\item[(ii)] If $\chi(1)>1$, then the restriction of $\chi$ to $P_n$ has a unique linear
constituent if and only
if $\chi(1)$ is odd and $n$ is a power of $2$.
\end{enumerate}
\end{teo}
 
We remark that Theorem \ref{thmB} shows that the only non-linear irreducible characters of symmetric groups having a unique linear constituent in restriction to a Sylow $2$-subgroup are exactly the odd-degree irreducible characters of $\fS_{2^n}$, for some natural number $n$. In this sense Theorem \ref{thmB} is a strong converse of the second statement of Theorem \ref{thmA}.

\section{Preliminaries}
In this section we fix the notation and we recall the main representation theoretic facts that we will need throughout the article. We refer the reader to \cite{Alperin} and \cite{NavBook} for a complete account of the theory. 

\subsection{Representations of wreath products}\label{subsec:wr}

Let $G$ be a finite group and let $M$ be a $\mathbb{C}G$-module. 
Let $M\boxtimes M$ be the two-fold tensor product of $M$ with itself. 
For all $(g_1, g_2 ; h)\in G\wr C_2$ and for all $m_1\otimes m_2 \in M\boxtimes M$, we let $(g_1, g_2 ; h)\cdot m_1\otimes m_2=g_1\cdot m_{h(1)}\otimes g_2\cdot m_{h(2)}.$
This yields a $\mathbb{C}(G\wr C_2)$-module structure on $M\boxtimes M$. 
We denote by $\widetilde{M\boxtimes M}$ the $\mathbb{C}(G\wr C_2)$-module 
obtained this way. 
It is clear that $\Res_{G\times G}(\widetilde{M\boxtimes M})\cong M\boxtimes M.$ 
If $\chi$ is the character afforded by $M$, then we denote by $\chi\times\chi$ and $\widetilde{\chi\times \chi}$ the characters afforded 
by $M\boxtimes M$ and $\widetilde{M\boxtimes M}$ respectively. 
For $J$ a $\mathbb{C}C_2$-module we denote by $\widehat{J}$ the $\mathbb{C}(G\wr C_2)$-module obtained by inflation of $J$. Similarly, if $\psi$ is the character afforded by $J$, then we denote by $\widehat{  \psi}$ the character afforded by $\widehat{J}$.

The following theorem characterizes the irreducible representations of $G\wr C_2$. It is a special case of the more general Theorem 4.4.3 of \cite{JK}. 

\begin{teo}\label{zoo}
A $\mathbb{C}(G\wr C_2)$-module $V$ is simple if and only if one of the following happens. 
\begin{enumerate}
\item $V\cong (\widetilde{M\boxtimes M})\bigotimes \widehat{J}$, for some simple $\mathbb{C}G$-module $M$ and for some simple $\mathbb{C}C_2$-module $J$, or

\smallskip

\item $V\cong \Ind_{G\times G}^{G\wr C_2}(L\boxtimes M)\cong \Ind_{G\times G}^{G\wr C_2}(M\boxtimes L)$, for some non-isomorphic simple $\mathbb{C}G$-modules $L$ and $M$. 
\end{enumerate}
In particular all these possibilities are mutually exclusive (no two of those representations are isomorphic).
\end{teo}

Denote by $K_1$ the cyclic group $C_2$ and define $K_n:=K_{n-1}\wr C_2$ for all natural numbers $n\geq 2$. 
Moreover, denote by $B_n$ the base group of $K_n$. Clearly $B_n\cong K_{n-1}\times K_{n-1}$.
The following lemma is easily deduced from Theorem \ref{zoo}.
\begin{lemma}\label{lemma:baserestriction}
Let $\chi\in \mathrm{Irr}(K_n)$. Then $\chi\res_{B_n}$ has a linear constituent if and only if $\chi(1)=1$ or $\chi=(\lambda\times \mu)\ind_{B_n}^{K_n}$, for some linear characters $\lambda, \mu\in \mathrm{Irr}(K_{n-1})$ such that $\lambda\neq \mu$. 

Moreover, $\chi$ is linear if and only if there exists a linear character $\lambda\in\mathrm{Irr}(K_{n-1})$ such that $\lambda\times \lambda$ is a constituent of $\chi\res_{B_n}$. 
\end{lemma}


\subsection{Partitions and characters of symmetric groups}\label{section:Sn}
In this section we summarize the main facts about the representation theory of symmetric groups that will be used in the rest of the paper. For a comprehensive account of this theory we refer the reader to \cite{James} and \cite{JK}.

Let $n$ be a natural number. A composition $\rho$ of $n$ is a finite sequence of non-negative integers $\rho=(\rho_1,\rho_2,\ldots , \rho_s),$ such that $\rho_1+\rho_2+\cdots +\rho_s=n$. We denote by $[\rho]$ the Young diagram corresponding to $\rho$. Let $\mathcal{P}(n)$ be the set of partitions of $n$. 
Given $m<n\in\mathbb{N}$ and $\lambda=(\lambda_1,\ldots, \lambda_k)\in\mathcal{P}(n), \mu=(\mu_1,\ldots, \mu_t)\in\mathcal{P}(m)$, we say that $\mu$ is a \textit{subpartition} of $\lambda$ if $t\leq k$ and $\mu_j\leq\lambda_j$ for all $j\in\{1,\ldots, t\}$. 
In this case we denote by $\lambda\smallsetminus\mu$ the composition of $n-m$ defined by 
$$\lambda\smallsetminus\mu=(\lambda_1-\mu_1,\lambda_2-\mu_2,\ldots,\lambda_t-\mu_t,\lambda_{t+1},\ldots,\lambda_k).$$
We will refer to $\lambda\smallsetminus\mu$ as a \textit{skew partition}. 

If $\lambda=(\lambda_1,\ldots, \lambda_k)$ and $\mu=(\mu_1,\ldots, \mu_t)$ are partitions of $n$, then we say that 
$\lambda > \mu$ if and only if the least $j\in\{1,\ldots, \mathrm{min}\{k,t\}\}$ such that $\lambda_j\neq \mu_j$ satisfies $\lambda_j>\mu_j$. This is called the \textit{dictionary order} on partitions. Clearly $>$ is a total order on $\mathcal{P}(n)$. 
Similarly, for $n_1,\ldots, n_k\in\mathbb{N}$ such that $n_1>n_2>\cdots >n_k\geq 1$, we can define a total order on $\mathcal{P}(n_1)\times\mathcal{P}(n_{2})\times\cdots\times\mathcal{P}(n_k)$ as follows. 
For $$\underline\alpha=(\alpha(1),\ldots ,\alpha(k)),\  \underline\beta=(\beta(1),\ldots ,\beta(k))\in \mathcal{P}(n_1)\times\mathcal{P}(n_{2})\times\cdots\times\mathcal{P}(n_k),$$
we say that $\underline\alpha > \underline\beta$ if the least $j\in\{1,\ldots, k\}$ such that $\alpha(j)\neq \beta(j)$ satisfies $\alpha(j)> \beta(j)$. This natural generalization of the dictionary order on partitions is a total order on the cartesian product $\mathcal{P}(n_1)\times\mathcal{P}(n_{2})\times\cdots\times\mathcal{P}(n_k)$.

Given $\lambda\in\mathcal{P}(n)$ we say that $j\in\{1,\ldots, n\}$ has \textit{multiplicity} $k$ in $\lambda$ if $\lambda$ has exactly $k$ parts of size $j$.
We equivalently denote a partition $\lambda=(\lambda_1,\ldots, \lambda_{\ell(\lambda)})$ by 
$$\lambda=(r_1^{\alpha_1},\ldots, r_k^{\alpha_k}),$$
if $\alpha_1+\cdots +\alpha_k=\ell(\lambda)$, $\lambda$ has exactly $\alpha_j$ parts of size $r_j$ for all $j\in\{1,\ldots, k\}$ and $r_j<r_{j+1}$ for all $j\in\{1,\ldots, k-1\}$. Here (and for the rest of the article) $\ell(\lambda)$ denotes the number of parts of $\lambda$.

Ordinary irreducible characters of the symmetric group $\fS_n$ are naturally labelled by partitions of $n$. Given any $\lambda\in \mathcal{P}(n)$ we denote by $\chi^\lambda$ the irreducible character of $\fS_n$ labelled by $\lambda$.  
Let $m<n$ be natural numbers. For any $\lambda$ and $\mu$ partitions of $m$ and $n-m$ respectively, the Littlewood-Richardson rule (see \cite[Chapter 16]{James}) describes the decomposition into irreducible constituents of the character $\chi$ of $\fS_n$ defined by $$\chi=(\chi^\lambda\times\chi^\mu)\ind_{\fS_m\times \fS_{n-m}}^{\fS_n}.$$

Since we will use it repeatedly in Section \ref{Sec:lastt}, for the reader's convenience we restate the rule below. In order to do this we need to recall a few technical definitions. 
\begin{definition}
Let $\lambda=(\lambda_1,\ldots,\lambda_k)\in \mathcal{P}(n)$ and let $\mathcal{C}=(c_1,\ldots,c_n)$ be a sequence of positive integers. We say that $\mathcal{C}$ is of \textit{type} $\lambda$ if 
$$|\{i\in\{1,2,\ldots , n\}\ :\ c_i=j\}|=\lambda_j,$$ 
for all $j\in\{1,\ldots, k\}$. We say that an element $c_j$ of $\mathcal{C}$ is \textit{good} if $c_j=1$ or if 
$$|\{i\in\{1,2,\ldots , j-1\}\ :\ c_i=c_j-1\}|>|\{i\in\{1,2,\ldots , j-1\}\ :\ c_i=c_j\}|.$$
Finally we say that the sequence $\mathcal{C}$ is \textit{good} if $c_j$ is good for every $j\in\{1,\ldots, n\}$. 
\end{definition}
\begin{example} The sequence $(1,2,3,1)$ is a good sequence of type $(2,1,1)$. On the other hand $(1,1,3,2)$ is not a good sequence since the number $3$ is not good. 
\end{example}

\begin{teo}\label{Thm:LR}
Let $n$ and $m$ be natural numbers such that $m<n$. Let $\mu\in \mathcal{P}(m)$ and $\nu\in\mathcal{P}(n-m)$. Then 
$$(\chi^\mu\times\chi^\nu)\ind_{\fS_m\times \fS_{n-m}}^{\fS_n}=\sum_{\lambda\in\mathcal{P}(n)}a_\lambda\chi^\lambda,$$
where $a_\lambda$ equals the number of ways to replace the nodes of $[\lambda\smallsetminus\mu]$ by natural numbers such that 
\begin{enumerate}
\item The sequence obtained by reading the natural numbers from right to left, top to bottom is a good sequence of type $\nu$.
\item The numbers are weakly increasing along rows.
\item The numbers are strictly increasing down the columns. 
\end{enumerate}
\end{teo}

For $k\in\{0,1,\ldots, n-1\}$ we say that the partition $(n-k,1^k)$ is a \textit{hook partition}. Moreover, we will denote by $\chi_n^k$ the irreducible character of $\fS_n$ labelled by $(n-k,1^k)$.
In the next sections we will be interested in studying irreducible characters of $\fS_n$ labelled by hook partitions. 
In particular, we will need the following special case of the Littlewood-Richardson rule. 

\begin{lemma}\label{lemma:LRhooks}
Let $n$ be a positive natural number and let $k\in\{0,1,\ldots, 2n-1\}$. If $k\leq n-1$, then 
$$(\chi_{2n}^k)\res_{\fS_n\times \fS_n}=\sum_{i=0}^{k-1} (\chi_n^i\times\chi_n^{k-1-i}) + \sum_{i=0}^k (\chi_n^i\times\chi_n^{k-i}).$$
If $k=j+n$ for some $j\in\{0,1,\ldots, n-1\}$, then 
$$(\chi_{2n}^k)\res_{\fS_n\times \fS_n}=\sum_{i=j}^{n-1} (\chi_n^i\times\chi_n^{n+j-i-1}) + \sum_{i=j+1}^{n-1} (\chi_n^i\times\chi_n^{n+j-i}).$$
\end{lemma}

\begin{proof}
We sketch the proof for the case $k\in\{0,1,\ldots, n-1\}$.
An easy application of the Littlewood-Richardson rule (Theorem \ref{Thm:LR}), shows that 
$$\Xi_k:=\sum_{i=0}^{k-1} (\chi_n^i\times\chi_n^{k-1-i}) + \sum_{i=0}^k (\chi_n^i\times\chi_n^{k-i})$$
is a constituent of $(\chi_{2n}^k)\res_{\fS_n\times \fS_n}$. 
Since $\chi_n^j(1)=\binom{n-1}{j}$ for all $j\in\{0,1,\ldots, n-1\}$, it follows by direct computation that $\chi_{2n}^k(1)=\Xi_k(1)$. Hence $\Xi_k=(\chi_{2n}^k)\res_{\fS_n\times \fS_n}$. 

A completely similar argument suffices to prove the case where $k=j+n$ for some $j\in\{0,1,\ldots, n-1\}$.
\end{proof}

\subsection{Sylow $2$-subgroups of $\fS_n$}\label{sec:syl}
 Let $P_2$ be the cyclic group $\langle (1,2)\rangle\leq \fS_2$. Let
further $P_1:=\{1\}$ and, for $d\geq 1$, we set 
$$P_{2^{d+1}}:=P_{2^d}\wr P_2:=\{(\sigma_1,\sigma_2;\pi): \sigma_1,\sigma_2\in P_{2^d},\, \pi\in P_2\}\,,$$
so that $P_{2^{d+1}}\cong K_{d+1}$.

We shall always identify $P_{2^d}$ with a Sylow
$2$-subgroup of $\fS_{2^d}$ in the usual way. That is, $(\sigma_1,\sigma_2;\pi)\in P_{2^d}$ is identified
with the element $\overline{(\sigma_1,\sigma_2;\pi)}\in \fS_{2^d}$ that is defined as follows: if $j\in\{1,\ldots,2^d\}$
is such that $j=b+2^{d-1}(a-1)$, for some $a\in\{1,2\}$ and some $b\in\{1,\ldots,2^{d-1}\}$ then
$$\overline{(\sigma_1,\sigma_2;\pi)}(j):=2^{d-1}(\pi(a)-1)+\sigma_{\pi(a)}(b).$$

\medskip

Now let $n\in \mathbb{N}$. Let $k_1>k_2>\cdots >k_t\geq 0$ be integers such that $n=2^{k_1}+\cdots +2^{k_t}$ is the $2$-adic expansion of $n$. 
By \cite[4.1.22, 4.1.24]{JK}, the Sylow $2$-subgroups of $\fS_n$ are
isomorphic to the direct product $\prod_{i=1}^t(P_{2^{k_i}})$. For subsequent computations it will be useful to
fix a particular Sylow $2$-subgroup $P_n$ of $\fS_n$ as follows: for $i\in\{1,\ldots ,t\}$ let $m(i):=\sum_{l=1}^{i-1}2^{k_l}$ and
$$P_{2^{k_i},m(i)}:= (1,1+m(i))\cdots (2^{k_i},2^{k_i}+m(i))\cdot P_{2^{k_i}}\cdot (1,1+m(i))\cdots (2^{k_i},2^{k_i}+m(i))\,.$$ 
Now set 
$$P_n:=P_{2^{k_1},m(1)}\times P_{2^{k_2},m(2)}\times\cdots\times P_{2^{k_t},m(t)}.$$
It is very easy to check that $N_{\fS_n}(P_n)=P_n$.
\section{Canonical bijections}

In this section we will prove Theorem \ref{thmA}. The proof is split into two parts (Theorems \ref{Thm:Alperin} and \ref{thm:bije} below). In particular, in Theorem \ref{Thm:Alperin} we first focus on the second statement of Theorem \ref{thmA}. 
For this reason, we start by restricting our attention to symmetric groups $\fS_{2^n}$, for all $n\in\mathbb{N}$. 
In the following lemma we characterize the odd-degree irreducible representations of $\fS_{2^n}$. 

\begin{lemma}\label{lemma:hook}
Let $\chi\in \mathrm{Irr}(\fS_{2^n})$. Then $\chi\in\mathrm{Irr}_{2'}(\fS_{2^n})$ if and only if $\chi=\chi^\lambda$, for some hook partition $\lambda$.
\end{lemma}

\begin{proof}
There are many ways to prove this. It is well known that $|\mathrm{Irr}_{2'}(\fS_{2^n})|=2^n$ (a proof can be found in \cite[Corollary 1.3]{Mac}). Moreover if $\lambda=(2^n-k,1^k)$ then $2$ does not divide $\chi^\lambda(1)$. This follows from the hook Formula for dimensions (see \cite[Page 77]{James}), since $$\chi^\lambda(1)=\frac{(2^n)!}{2^n\cdot k!(2^n-1-k)!}=\frac{(2^n-1)(2^n-2)\cdots (2^n-k)}{k!},$$
and this is clearly an odd integer.
\end{proof}

For the reader convenience, we modify the notation previously introduced in section \ref{section:Sn}. For $n\in\mathbb{N}$ and $k\in\{0,1,\ldots, 2^n-1\}$, we will now denote by 
$\chi_n^k$ the irreducible character of $\fS_{2^n}$ labelled by the partition $(2^n-k,1^k)$.

\begin{teo}\label{Thm:Alperin}
Let $\chi\in \mathrm{Irr}_{2'}(\fS_{2^n})$. 
Then $\chi\res_{P_{2^n}}$ has a unique linear constituent $\Phi_{n}(\chi)$.

Moreover, the map $$\Phi_{n}: \mathrm{Irr}_{2'}(\fS_{2^n})\longrightarrow \mathrm{Irr}_{2'}(P_{2^n}),$$ defined by $\chi\mapsto \Phi_{n}(\chi)$ is a bijection. 
\end{teo}

\begin{proof}
We proceed by induction on $n$. For $n=1$ we have that $\fS_2=P_2$, and the statement holds trivially. 
Let $n\geq 2$ and assume that the statement holds for $n-1$. 
Let $\chi\in\mathrm{Irr}_{2'}(\fS_{2^n})$. By Lemma \ref{lemma:hook}, there exists $k\in \{0,1,\ldots, 2^n-1\}$ such that $\chi=\chi_n^k$.
From Section \ref{sec:syl},
it is easy to see that $P_{2^n}=B_n\rtimes \left\langle g\right\rangle$, where $$g=\prod_{i=1}^{2^{n-1}}(i, i+2^{n-1}),$$
and $B_n=P_{2^{n-1}}\times P_{2^{n-1}}^g\cong P_{2^{n-1}}\times P_{2^{n-1}}$. Let $T_n$ be the maximal Young subgroup of $\fS_{2^n}$ containing $B_n$. Clearly $T_n=\fS_{2^{n-1}}\times \fS_{2^{n-1}}$. 

By Lemma \ref{lemma:LRhooks} we have that 
$\chi_n^k\res_{T_n}$ equals

\begin{equation*}
\begin{cases}
\sum_{j=0}^{k-1} (\chi_{n-1}^j\times\chi_{n-1}^{k-1-j}) + \sum_{j=0}^k (\chi_{n-1}^j\times\chi_{n-1}^{k-j}) &\text{ if } 0\leq k\leq 2^{n-1}-1\,,\\

\\

\sum_{i=k-2^{n-1}}^{2^{n-1}-1} (\chi_{n-1}^i\times\chi_{n-1}^{k-1-i}) + \sum_{i=k-2^{n-1}+1}^{2^{n-1}-1} (\chi_{n-1}^i\times\chi_{n-1}^{k-i}) &\text{ if } 2^{n-1}\leq k\leq 2^{n}-1\,.\\
\end{cases}
\end{equation*}

\noindent Therefore, using the inductive hypothesis we deduce that $\chi_n^k\res_{B_n}$ is equal to

\begin{equation*}
\begin{cases}
\sum_{j=0}^{k-1} (\phi_{n-1}^j\times\phi_{n-1}^{k-1-j}) + \sum_{j=0}^k (\phi_{n-1}^j\times\phi_{n-1}^{k-j})+\Delta_k &\text{ if } 0\leq k\leq 2^{n-1}-1\,,\\

\\

\sum_{i=k-2^{n-1}}^{2^{n-1}-1} (\phi_{n-1}^i\times\phi_{n-1}^{k-1-i}) + \sum_{i=k-2^{n-1}+1}^{2^{n-1}-1} (\phi_{n-1}^i\times\phi_{n-1}^{k-i})+\Delta_k &\text{ if } 2^{n-1}\leq k\leq 2^{n}-1\,,\\
\end{cases}
\end{equation*}


\noindent where $\phi_{n-1}^j=\Phi_{n-1}(\chi_{n-1}^j)$ for all $j\in \{0,1,\ldots, 2^{n-1}-1\}$ and $\Delta_k$ is a (possibly empty) sum of irreducible characters of $P_{2^{n-1}}\times P_{2^{n-1}}$ of even degree, for all $k\in\{0,\dots 2^n-1\}$. 

For $k\in \{0,1,\ldots, 2^n-1\}$, let $\gamma(k)$ be $\frac{k}{2}$ (respectively $\frac{k-1}{2}$) if $k$ is even (respectively odd).  
Then 

$$(\chi_n^k\res_{P_{2^n}})\res_{B_n}=(\chi_n^k\res_{T_n})\res_{B_n}= (\phi_{n-1}^{\gamma(k)}\times \phi_{n-1}^{\gamma(k)}) + \Gamma_k + \Delta_k, $$
where $\Gamma_k$ is a sum of linear characters of $B_n$ of the form
$\alpha\times \beta$, where $\alpha$ and $\beta$ are non-isomorphic linear characters of $\fS_{2^{n-1}}$. 
By Lemma \ref{lemma:baserestriction}, we deduce that $\chi^k_n\res_{P_{2^n}}$ has a unique linear constituent $\Phi_n(\chi^k_n)$. Moreover, 
$$\Phi_n(\chi^k_n)\res_{B_n}=\phi_{n-1}^{\gamma(k)}\times \phi_{n-1}^{\gamma(k)}.$$ 

To conclude the proof, we just need to show that whenever $k_1<k_2$ are such that $\gamma(k_1)=\gamma(k_2)$, then $\Phi_n(\chi^{k_1}_n)\neq \Phi_n(\chi^{k_2}_n)$. 
Clearly $\gamma(k_1)=\gamma(k_2)$ if and only if $k_1=2j$ and $k_2=2j+1$, for some $j\in\{0,1,\ldots, 2^{n-1}-1\}$.
Let $g\in P_{2^n}$ be a $2^n$-cycle.
By the Murnaghan-Nakayama rule (see \cite[Page 79]{James}), we have that 
$$\chi_n^{k_1}(g)=1\ \ \text{and}\ \ \chi_n^{k_2}(g)=-1.$$
Moreover, it is very easy to show that $\theta(g)=0$ for all $\theta\in \mathrm{Irr}(P_{2^n})$ such that $\theta(1)$ is even. 
\color{black}
Hence $$\Phi_n(\chi^{k_1}_n)(g)=\chi_n^{k_1}(g)=1\neq -1=\chi_n^{k_2}(g)=\Phi_n(\chi^{k_2}_n)(g).$$
The proof is concluded. 
\end{proof}

Building on the bijection described in Theorem \ref{Thm:Alperin},  we are now ready to define a canonical bijection between ${\rm Irr}_{2'}(\fS_n)$ and ${\rm Irr}_{2'}(P_n)$, for all $n\in\mathbb{N}$. This will complete the proof of Theorem \ref{thmA}.

\begin{teo}\label{thm:bije}
Let $n\in\mathbb{N}$. There exists a natural bijection 
$$\Psi_n: \mathrm{Irr}_{2'}(\fS_{n})\longrightarrow \mathrm{Irr}_{2'}(P_n).$$
If $n=2^k$ for some $k\in\mathbb{N}$, then $\Psi_n=\Phi_{k}$. 
\end{teo}
\begin{proof}
Let $K$ be the subset of $\mathcal{P}(n)$ defined by 
$K=\{\lambda\in\mathcal{P}(n)\ :\ \chi^\lambda(1)\ \ \text{is odd}\}.$
The dictionary order, defined in Section \ref{section:Sn}, restricts to a total order on $K$. 
Let $K=\{\lambda_1,\lambda_2,\ldots ,\lambda_{|K|}\}$ be such that $\lambda_i>\lambda_{i+1}$ for all $i\in\{1,\ldots,|K|-1\}$.

Let $n=2^{k_1}+2^{k_2}+\cdots +2^{k_t}$ be the $2$-adic expansion of $n$. From Section \ref{sec:syl}, we have that $P_{n}\cong P_{2^{k_1}}\times P_{2^{k_2}}\times\cdots\times P_{2^{k_t}}.$ 
If $\psi$ is a linear character of $P_n$ then $\psi=\psi_1\times\cdots\times \psi_t$, where $\psi_j$ is a linear character of $P_{2^{k_j}}$ for all $j\in\{1,\ldots, t\}$. 
By Theorem \ref{Thm:Alperin} we have that for all $j\in\{1,\ldots, t\}$, there exists a unique hook partition $\alpha(j)\in\mathcal{P}(2^{k_j})$, such that  $\psi_j=\Phi_{k_j}(\chi^{\alpha(j)})$.
In this case we denote $\psi$ by $\psi^{\underline{\alpha}}$, where 
$$\underline{\alpha}=(\alpha(1),\ldots, \alpha(t))\in \mathcal{P}(2^{k_1})\times\cdots\times \mathcal{P}(2^{k_t}).$$
Let $H$ be the subset of $\mathcal{P}(2^{k_1})\times\cdots\times \mathcal{P}(2^{k_t})$ consisting of all the elements $(\alpha(1),\ldots, \alpha(t))$, such that $\alpha(j)\in\mathcal{P}(2^{k_j})$ is a hook partition. 
The total order defined in Section \ref{section:Sn} restricts to a total order on $H$. Let $$H=\{\underline{\eta_1},\underline{\eta_2},\ldots ,\underline{\eta_{|H|}}\},$$ be such that $\underline{\eta_i}>\underline{\eta_{i+1}}$ for all $i\in\{1,\ldots,|H|-1\}$.
It is now clear that $$\mathrm{Irr}_{2'}(P_n)=\{\psi^{\underline{\eta_j}}\ :\ j\in\{1,\ldots,|H|\}\}.$$
By \cite[Proposition 4.4]{Olsson} we have that $|K|=|H|$. Hence the map 
$$\Psi_n: \mathrm{Irr}_{2'}(\fS_{n})\longrightarrow \mathrm{Irr}_{2'}(P_n),$$
defined by $\Psi_n(\chi^{\lambda_i})=\psi^{\underline{\eta_i}}$ for all $i\in\{1,\ldots, |K|\}$ is a choice-free canonical bijection.  
It is clear by construction that $\Psi_{2^k}=\Phi_k$, for all $k\in\mathbb{N}$. 
\end{proof}

In the following example we concretely construct the bijection $\Psi_6$. 

\begin{example}
Let $n=6$ and let $P_6$ be a Sylow $2$-subgroup of $\fS_6$. Then $P_6\cong P_4\times P_2\cong C_2\wr C_2\times C_2$. 
Let 
\begin{eqnarray*}
\mathcal{K}_2 &=& \{(2), (1^2)\}\subseteq \mathcal{P}(2),\\
\mathcal{K}_4 &=& \{(4), (3,1), (2,1^2), (1^4)\}\subseteq \mathcal{P}(4),\ \text{and}\\
\mathcal{K}_6 &=& \{(6), (5,1), (4,2), (3^2), (2^3), (2^2,1^2), (2,1^4), (1^6)\}\subseteq \mathcal{P}(6).
\end{eqnarray*}
Then $\mathrm{Irr}_{2'}(\fS_n)=\{\chi^\lambda\ :\ \lambda\in\mathcal{K}_n\}$ for all $n\in\{2,4,6\}$. Notice that all the partitions listed in the above sets are written from left to right according to the dictionary order. Let $\mathcal{H}$ be the subset of $\mathcal{P}(4)\times\mathcal{P}(2)$ defined by 
$$\mathcal{H}=\{((4),(2)), ((4),(1^2)), ((3,1),(2)), \ldots, ((1^4),(2)), ((1^4),(1^2))\}.$$
For $\underline{\alpha}=(\alpha(1),\alpha(2))\in\mathcal{H}$ let 
$\psi^{\underline{\alpha}}$ be the linear character of $P_6$ defined by 
$$\psi^{\underline{\alpha}}=
\Phi_2(\chi^{\alpha(1)})\times\Phi_1(\chi^{\alpha(2)}).$$
Then $\mathrm{Irr}_{2'}(P_6)=\{\psi^{\underline{\alpha}}\ :\ \underline{\alpha}\in\mathcal{H}\}$. 

The bijection $\Psi_6$ is constructed as follows: 
\begin{eqnarray*}
\chi^{(6)} & \longmapsto & \psi^{((4),(2))},\\
\chi^{(5,1)} & \longmapsto & \psi^{((4),(1^2))},\\
\chi^{(4,2)} & \longmapsto & \psi^{((3,1),(2))},\\
\chi^{(3^2)} & \longmapsto & \psi^{((3,1),(1^2))},\\
\chi^{(2^3)} & \longmapsto & \psi^{((2,1^2),(2))},\\
\chi^{(2^2,1^2)} & \longmapsto & \psi^{((2,1^2),(1^2))},\\
\chi^{(2,1^4)} & \longmapsto & \psi^{((1^4),(2))},\\
\chi^{(1^6)} & \longmapsto & \psi^{((1^4),(1^2))}.\\
\end{eqnarray*}
\end{example}

Let $n=2^k$ and let $\chi^\lambda\in \mathrm{Irr}_{2'}(\fS_{2^k})$. By Theorem \ref{Thm:Alperin} we have that $\Psi_{2^k}(\chi^\lambda)=\Phi_k(\chi^\lambda)$ is a constituent of the restriction of the character $\chi^\lambda$ to a Sylow $2$-subgroup of $\fS_{2^k}$.  
In general the bijection $\Psi_n$ does not \textit{agree} with restriction.
For example, direct computations show that when $n=6$ we have that $$\chi^{(3,3)}\res_{P_6}=\psi^{((3,1),(2))}+\psi^{((4),(2))}+\psi^{((1^4),(1^2))}+\Delta,$$
where $\Delta$ is an irreducible character of $P_6$, such that $\Delta(1)=2$. Hence $\Psi_6(\chi^{(3,3)})=\psi^{((3,1),(1^2))}$ is not a constituent of $\chi^{(3,3)}\res_{P_6}$. 

Anyway, in the following proposition we show that the bijection $\Psi_n$ is induced by restriction when $n=2^m+1$. 

\begin{proposition}\label{2^n+1}
Let $n=2^m+1$ and let $\lambda\in\mathcal{P}(n)$ be such that  $\chi^\lambda(1)$ is odd. Then $\Psi_n(\chi^\lambda)$ is a constituent of $\chi^\lambda\res_{P_n}$. 
\end{proposition}
\begin{proof}
Let $\mathcal{A}$ be the subset of $\mathcal{P}(n)$ defined by 
$$\mathcal{A}=\{(n-(k+1),2,1^{k-1})\ :\ k\in\{1,\ldots,n-3\}\} \cup \{(n), (1^n)\}.$$
We first show that $\mathrm{Irr}_{2'}(\fS_{n})=\{\chi^\lambda\ :\ \lambda\in \mathcal{A}\}$. 
Since $P_n\cong P_{2^m}$ we have that $$|\mathrm{Irr}_{2'}(\fS_{n})|=|\mathrm{Irr}_{2'}(P_n)|=|\mathrm{Irr}_{2'}(P_{2^m})|=2^m=|\mathcal{A}|.$$
Hence it is enough to show that $\chi^\lambda(1)$ is odd for all $\lambda\in \mathcal{A}$. 
For $k\in\{1,\ldots,n-3\}$, let $\lambda_k=(n-(k+1),2,1^{k-1})$ and let $\lambda_k^*=(n-(k+1),1^{k})\in\mathcal{P}(2^m)$. 
By the hook lenght formula we have that 
\begin{eqnarray*}
\chi^{\lambda_k}(1) &=& \frac{(2^m+1)!}{(k-1)!(2^m-(k+3))!(k+1)(2^m-(k+1))2^m}\\
\\
&=& \frac{(2^m+1)k(2^m-(k+1))}{(k+1)(2^m-k)}\cdot\chi^{\lambda_k^*}(1).\\
\end{eqnarray*}
Since $\chi^{\lambda_k^*}(1)$ is odd by Lemma \ref{lemma:hook}, it follows that $\chi^{\lambda_k}(1)$ is odd. 

Moreover, by the Branching rule we have that $\chi^{\lambda_k^*}$ is a constituent of $\chi^{\lambda_k}\res_{\fS_{n-1}\times \fS_1}$. Hence $\Phi_m(\chi^{\lambda_k^*})\times \Phi_0(\chi^{(1)})$ is a constituent of 
$\chi^{\lambda_k}\res_{P_{n}}$. 
It remains to show that $\Psi_n(\chi^{\lambda_k})=\Phi_m(\chi^{\lambda_k^*})\times \Phi_0(\chi^{(1)})$. This follows by observing that $\lambda_k > \lambda_j$ if and only if $k<j$ and therefore if and only if $\lambda_k^* > \lambda_j^*$. 
\end{proof}

\section{More on linear constituents}\label{Sec:lastt}
This section is devoted to the proof of Theorem \ref{thmB}.
Adopting the notation introduced in Section \ref{section:Sn}, let $\lambda\in\mathcal{P}(2n)$ be such that $\lambda=(\lambda_1,\ldots, \lambda_{\ell(\lambda)})=(r_1^{\alpha_1},\ldots,r_t^{\alpha_t}).$

\begin{definition}\label{def:delta}
Let $\lambda\in\mathcal{P}(2n)$ be as above. 
Call each $r_i^{\alpha_i}$ a \textit{cluster} of $\lambda$, and call $r_i^{\alpha_i}$ an \textit{odd cluster} if $r_i$ and $\alpha_i$ are odd. Since $\lambda\in\mathcal{P}(2n)$ we have that $\lambda$ has an even number of odd clusters. Let $\Delta(\lambda)$ be the partition of $n$ obtained from $\lambda$ as follows: 

\noindent Replace the cluster $r_i^{\alpha_i}$ in $\lambda$ by

\begin{equation*}
\begin{cases}
(\frac{r_i}{2})^{\alpha_i} &\text{ if } r_{i} \ \text{is even}\,,\\
\big((\frac{r_i+1}{2})^{\frac{\alpha_i}{2}}, (\frac{r_i-1}{2})^{\frac{\alpha_i}{2}}\big) &\text{ if } r_i\ \text{is odd and}\ \alpha_i\ \text{is even}\,,\\
\big((\frac{r_i+1}{2})^{\frac{\alpha_i-1}{2}}, \frac{r_i+\varepsilon_i}{2}, (\frac{r_i-1}{2})^{\frac{\alpha_i-1}{2}}\big) &\text{ if } r_i\ \text{and}\ \alpha_i\ \text{are odd}\,,\\
\end{cases}
\end{equation*}
where $\varepsilon_i=1$ if $r_i^{\alpha_i}$ is the 1st, 3rd, 5th, ... odd cluster in $\lambda$; and $\varepsilon_i=-1$ if $r_i^{\alpha_i}$ is the 2nd, 4th, 6th, ... odd cluster in $\lambda$.

Moreover denote by $S(\lambda)$ the skew partition defined by $S(\lambda)=\lambda\smallsetminus\Delta(\lambda).$
\end{definition}

The following example should clarify the idea behind the above definition. 

\begin{example}
Let $\lambda=(7,7,7,6,5,4,3,3,1,1)=(7^3, 6,5,4,3^2, 1^2)\in\mathcal{P}(44)$.
Then the odd clusters of $\lambda$ are $7^3$ and $5^1$.  
Moreover we have that $$\Delta(\lambda)=(\underbrace{4,4,3}_{7},\underbrace{3}_{6},\underbrace{2}_{5},\underbrace{2}_4, \underbrace{2,1}_3,\underbrace{1,0}_1)\in\mathcal{P}(22),$$ 
where the labelling below each part of $\Delta(\lambda)$ records the corresponding cluster of $\lambda$. 
In this situation we have 
$$S(\lambda)=(3,3,4,3,3,2,1,2,0,1).$$
Figure 1 below highlights the relation between $\lambda, \Delta(\lambda)$ and $S(\lambda)$. 
\begin{figure}[!htp]
\begin{center}

\begin{tikzpicture}[scale=0.4, every node/.style={transform shape}]
\tikzstyle{every node}=[font=\large]

\draw (0,16)--(14,16);
\draw (0,14)--(14,14);
\draw (0,12)--(14,12);
\draw (0,10)--(14,10);
\draw (0,8)--(12,8);
\draw (0,6)--(10,6);
\draw (0,4)--(8,4);
\draw (0,2)--(6,2);
\draw (0,0)--(6,0);
\draw (0,-2)--(2,-2);
\draw (0,-4)--(2,-4);
\draw (0,16)--(0,-4);
\draw (2,16)--(2,-4);
\draw (4,16)--(4,0);
\draw (6,16)--(6,0);
\draw (8,16)--(8,4);
\draw (10,16)--(10,6);
\draw (12,16)--(12,8);
\draw (14,16)--(14,10);
\node at (9,15) {$\underline{1}$};
\node at (11,15) {$\underline{1}$};
\node at (13,15) {$\underline{1}$};
\node at (9,13) {$\underline{2}$};
\node at (11,13) {$\underline{2}$};
\node at (13,13) {$\underline{2}$};
\node at (9,11) {$\underline{3}$};
\node at (11,11) {$\underline{3}$};
\node at (13,11) {$\underline{3}$};
\node at (9,9) {$\underline{4}$};
\node at (11,9) {$\underline{4}$};
\node at (7,9) {$\underline{4}$};
\node at (9,7) {$\underline{5}$};
\node at (7,7) {$\underline{5}$};
\node at (5,5) {$\underline{6}$};
\node at (7,5) {$\underline{6}$};
\node at (5,3) {$\underline{7}$};
\node at (5,1) {$\underline{8}$};
\node at (7,11) {$1$};
\node at (5,7) {$2$};
\node at (3,1) {$7$};
\node at (1,-3) {$9$};
\node at (1,15) {$\bullet$};
\node at (3,15) {$\bullet$};
\node at (5,15) {$\bullet$};
\node at (7,15) {$\bullet$};
\node at (1,13) {$\bullet$};
\node at (3,13) {$\bullet$};
\node at (5,13) {$\bullet$};
\node at (7,13) {$\bullet$};
\node at (1,11) {$\bullet$};
\node at (1,9) {$\bullet$};
\node at (1,7) {$\bullet$};
\node at (1,5) {$\bullet$};
\node at (1,3) {$\bullet$};
\node at (1,1) {$\bullet$};
\node at (1,-1) {$\bullet$};
\node at (3,11) {$\bullet$};
\node at (3,9) {$\bullet$};
\node at (3,7) {$\bullet$};
\node at (3,5) {$\bullet$};
\node at (3,3) {$\bullet$};
\node at (3,1) {};
\node at (5,11) {$\bullet$};
\node at (5,9) {$\bullet$};
\node at (7,11) {};
\node at (7,9) {};
\node at (9,11) {};
\node at (9,9) {};
\node at (-2,6){$[\lambda]=$};

\end{tikzpicture}

\end{center}
\caption{The Young diagram $[\Delta(\lambda)]$ consists of those boxes containing a black dot. The Young diagram $[S(\lambda)]$ of the skew partition $S(\lambda)$ consists of the remaining boxes, containing numbers.
\\
Underlined numbers correspond to Step 1 of the algorithm described in the proof of Proposition \ref{prop:DelS}. Non-underlined numbers are assigned to the boxes of $[S(\lambda)]$ according to Step 2 of the algorithm described in the proof of Proposition \ref{prop:DelS}.}
\label{fig:a}
\end{figure}
\vspace{.2cm}

\end{example}

\begin{proposition}\label{prop:DelS}
Let $\lambda\in \mathcal{P}(2n)$ and let $\Delta(\lambda)\in\mathcal{P}(n)$ be as in Definition \ref{def:delta}. Then $\chi^{\Delta(\lambda)}\times\chi^{\Delta(\lambda)}$ is a constituent of $\chi^\lambda\downarrow_{\fS_n\times \fS_n}$.
\end{proposition}

\begin{proof}
By the Littlewood-Richardson rule (Theorem \ref{Thm:LR}) and Frobenius reciprocity, it is enough to show that there exists a way of replacing the nodes of $[S(\lambda)]=[\lambda\smallsetminus\Delta( \lambda)]$ by integers such that conditions (1), (2) and (3) of Theorem \ref{Thm:LR} are satisfied. 
From the definition of $\Delta(\lambda)$ it is easy to check that if $\lambda_j$ is even then $S(\lambda)_j=\Delta(\lambda)_j$. On the other hand we have that $S(\lambda)_j=\Delta(\lambda)_j\pm 1$, when $\lambda_j$ is odd. 
(Notice that we allow $S(\lambda)$ to have parts of size $0$, in order to have $\ell(S(\lambda))=\ell(\lambda)$). 
Below we describe (in two steps) an algorithm to replace the nodes of $[S(\lambda)]$ by integers. 

\smallskip

\noindent\textbf{Step 1.} Let $j\in\{1,2,\ldots, \ell(\lambda)\}$. 

\noindent\textbf{(a)} If $S(\lambda)_j=\Delta(\lambda)_j$, then replace every node of row $j$ of $[S(\lambda)]$ by $j$. 

\noindent\textbf{(b)} If $S(\lambda)_j=\Delta(\lambda)_j-1$, then replace every node of row $j$ of $[S(\lambda)]$ by $j$. 

\noindent\textbf{(c)} If $S(\lambda)_j=\Delta(\lambda)_j+1$, then leave the most left node of row $j$ of $[S(\lambda)]$ empty and replace every other node of row $j$ of $[S(\lambda)]$ by $j$. 

\noindent(In figure \ref{fig:a} the nodes of $[S(\lambda)]$ are replaced by underlined numbers according to Step 1 of our algorithm.)  

\noindent\textbf{Remarks on Step 1.}

\noindent$\bullet$  If $\Delta(\lambda)_j=1$ and $S(\lambda)_j=0$ then $[S(\lambda)]$ does not have nodes in row $j$, hence in $(b)$ we are doing nothing. Similarly, If $\Delta(\lambda)_j=0$ and $S(\lambda)_j=1$ then the unique node  in row $j$ of $[S(\lambda)]$ is the most left one, hence also in this case in $(c)$ we are doing nothing. 

\smallskip

\noindent$\bullet$ Notice that, once Step 1 is completed, we have that for all $j\in\{1,2,\ldots, \ell(\lambda)\}$, the integer $j$ appears only in row $j$ of $[S(\lambda)]$. Therefore at the moment our numbers are weakly increasing along rows and obviously strictly increasing down the columns. Moreover the sequence $\sigma$ obtained by reading the numbers we distributed in $[S(\lambda)]$ from right to left, top to bottom contains exactly $n_j$ entries equal to $j$, where 
\begin{equation*}
n_j= \begin{cases}
\Delta(\lambda)_j &\text{ or },\\
\Delta(\lambda)_j-1\,.\\
\end{cases}
\end{equation*}

Since each $j$ appears only in row $j$, we have that every $j$ appears before any $j+1$ in $\sigma$. Moreover, if $n_j=\Delta(\lambda)_j$ then clearly $n_j\geq n_{j+1}$. On the other hand, if $n_j=\Delta(\lambda)_j-1$ and $n_{j+1}=\Delta(\lambda)_{j+1}-1$ again we have $n_j\geq n_{j+1}$. Finally, if $n_j=\Delta(\lambda)_j-1$ and $n_{j+1}=\Delta(\lambda)_{j+1}$ then $S(\lambda)_j=\Delta(\lambda)_j-1=n_j$. Hence $\lambda_j=2n_j+1$. 
Similarly $2n_{j+1}\leq\lambda_{j+1}\leq 2n_{j+1}+1$, hence $n_j\geq n_{j+1}$, because $\lambda$ is a partition. 
This proves that $\sigma$ is a good sequence.

\smallskip

\noindent$\bullet$ If $\lambda_j$ is even for all $j\in\{1,2,\ldots, \ell(\lambda)\}$ then the algorithm ends. In fact, by the above remark, it follows that Step 1 is enough to obtain a way to replace the nodes of $[S(\lambda)]$ with integers such that: 
the sequence obtained by reading the natural numbers from right to left, top to bottom is a good sequence of type $\Delta(\lambda)$;
the numbers are weakly increasing along rows;
the numbers are strictly increasing down the columns. 

\bigskip

Let $2q$ be the number of odd parts of $\lambda$ and suppose
that $q\neq 0$ (i.e. $\lambda$ has odd parts). In order to obtain a sequence of type $\Delta(\lambda)$ we need to replace the remaining $q$ nodes of $[S(\lambda)]$ with integers $z_1<z_2<\cdots<z_q$ such that $z_i\in \{1,2,\ldots, \ell(\lambda)\}$ for all $i\in\{1,2,\ldots, q\}$. 
By Step 1, we have that no two of the remaining $q$ empty nodes lie in the same row of $[S(\lambda)]$, since they are the most left nodes of $q$ distinct rows. Let $k_1<k_2<\cdots<k_q$ be the rows of $[S(\lambda)]$ containing an empty (most left) node. 
By construction, for all $i\in \{1,2,\ldots, q\}$ we have that $z_i<k_i$. 

\smallskip

\noindent\textbf{Step 2.}
If $q\neq 0$, replace the most left node of row $k_i$ by $z_i$, for all $i\in\{1,2,\ldots, q\}$. 

\noindent(In figure \ref{fig:a} the nodes of $[S(\lambda)]$ are replaced by non-underlined numbers according to Step 2 of our algorithm.)  

\smallskip

\noindent\textbf{Remarks on Step 2.} The overall configuration obtained after steps 1 and 2 has the following properties. 

\noindent$\bullet$ Let $\tau=(\tau_1,\ldots, \tau_n)$ be the sequence obtained by reading the numbers from right to left, top to bottom. 
Let $j\in\{2,\ldots, \ell(\Delta(\lambda))\}$ and suppose that $\tau_x=j$ for some $x\in\{1,2,\ldots, n\}$. We want to show that $\tau_x$ is good. Let $n_j$ denote the number of $j$s lying in row $j$ of $[S(\lambda)]$. By constructuion this number did not change after Step 2. Hence, from the remarks after Step 1, we deduce that $n_{j-1}\geq n_j$. Therefore we have that every $j$ lying in row $j$ of $[S(\lambda)]$ is good in $\tau$. 
Suppose now that $j$ lies in row $k$ of $[S(\lambda)]$. Then necessarily $k=k_i>z_i=j$ for some $i\in\{1,\ldots, q\}$ and $j=z_i$ lies in the most left node of row $k=k_i$ of $[S(\lambda)]$. 
In this case we have that 
$$|\{y\in\{1,\ldots , x-1\}\ :\ \tau_y=j\}|=\Delta(\lambda)_j-1.$$
Again, we need to distinguish two possibilities. 
If all $(j-1)$s are in row $(j-1)$ of $[S(\lambda)]$ then 
$$|\{1\leq y<x\ :\ \tau_y=j-1\}|=\Delta(\lambda)_{j-1}>\Delta(\lambda)_{j}-1=|\{1\leq y<x\ :\ \tau_y=j\}|,$$
since $j-1<j<k$. Hence $\tau_x=j$ is good in $\tau$. 
On the other hand, if $z_{i-1}=j-1$ then we have exactly $\Delta(\lambda)_{j-1}-1$ integers equal to $(j-1)$ lying in row $(j-1)$ of $[S(\lambda)]$ and one $(j-1)$ lying in row $k_{i-1}$. Since $k_{i-1}<k_i$, we have again that $$|\{1\leq y<x\ :\ \tau_y=j-1\}|=\Delta(\lambda)_{j-1}>\Delta(\lambda)_{j}-1=|\{1\leq y<x\ :\ \tau_y=j\}|.$$
We conclude that $\tau_x$ is a good element of $\tau$ and therefore that $\tau$ is a good sequence of type $\Delta(\lambda)$. 

\smallskip

\noindent$\bullet$ The numbers are clearly weakly increasing along the rows, since $z_i<k_i$ for all $i\in\{1,\ldots, q\}$.  

\smallskip

\noindent$\bullet$ The numbers are strictly increasing down the columns. Since $\Delta(\lambda)$ is a partition, this follows by observing that for all $i\in\{1,\ldots, q\}$, $z_i$ is either the highest entry in its column, or it lies below $z_h$, for some $1\leq h\leq i-1$. 

\smallskip

By the Littlewood-Richardson rule (Theorem \ref{Thm:LR}), the proof is now complete. 
\end{proof}

\begin{corollary}\label{cor:11}
Let $\lambda\in\mathcal{P}(2^n)$. Then $\chi^\lambda\res_{P_{2^n}}$ has a linear constituent.
\end{corollary}
\begin{proof}
Proceed by induction on $n$. If $n=0,1$ then the result follows easily by direct computation. 
Suppose that $n\geq 2$. By Proposition \ref{prop:DelS} there exists $\Delta(\lambda)\in \mathcal{P}(2^{n-1})$ such that $\chi^{\Delta(\lambda)}\times\chi^{\Delta(\lambda)}$ is a constituent of $\chi^\lambda\res_{\fS_{2^{n-1}}\times \fS_{2^{n-1}}}$. 
By inductive hypothesis there exists a linear character $\phi$ of $P_{2^{n-1}}$ such that $\phi$ is a constituent of $\chi^{\Delta(\lambda)}\res_{P_{2^{n-1}}}$. 
Therefore $\phi\times\phi$ is a linear constituent of $\chi^\lambda\res_{P_{2^{n-1}}\times P_{2^{n-1}}}$. By Lemma \ref{lemma:baserestriction} we deduce that $\chi^\lambda\res_{P_{2^n}}$ has a linear constituent. 
\end{proof}

In order to prove Theorem \ref{thmB}, we will need the following technical lemmas. 

\begin{lemma}\label{lem:nightmare}
Let $n=2^k+2^t$ for some natural numbers $k$ and $t$ such that $k>t\geq 0$. Let $\lambda\in\mathcal{P}(n)$ be such that  $\chi^\lambda(1)$ is odd and $\chi^\lambda(1)>1$. Then $\chi^\lambda\res_{\fS_{2^k}\times \fS_{2^t}}$ is not irreducible. 
\end{lemma}
\begin{proof}
Suppose that there exists $\mu\in\mathcal{P}(2^k)$ and $\nu\in\mathcal{P}(2^t)$ such that $$\chi^\lambda\res_{\fS_{2^k}\times \fS_{2^t}}=\chi^\mu\times\chi^\nu.$$ 
Since $\chi^ \lambda(1)$ is odd, we deduce that also $\chi^\mu(1)$ and $\chi^\nu(1)$ are odd integers. Therefore by Lemma \ref{lemma:hook} there exists $i\in\{0,1,\ldots, 2^k-1\}$ and $j\in\{0,1,\ldots, 2^t-1\}$ such that 
$$\mu=(2^k-i,1^i)\ \ \text{and}\ \ \nu=(2^t-j,1^j).$$
We split the proof into two different cases. 

\smallskip

\textbf{(1)} Suppose that $j\notin\{0,2^t-1\}$. 
Let $g\in \fS_n$ be a $(2^t-1)$-cycle such that $g(x)=x$ for all $x\in\{2^k+1,2^k+2,\ldots ,n\}$, and let $h\in \fS_n$ be the $(2^t-1)$-cycle defined $$h=(2^k+1,2^k+2,\ldots ,n-1).$$
Clearly $(g,1), (1,h)\in \fS_{2^k}\times \fS_{2^t}$ and $$\chi^\mu(g)\chi^\nu(1)=\chi^\lambda(g)=
\chi^\lambda(h)=\chi^\mu(1)\chi^\nu(h),$$
since $g$ and $h$ are conjugate in $\fS_n$.
Moreover, from the Murnaghan-Nakayama rule it is easy to deduce that $\chi^\nu(h)=0$, since $[\nu]$ does not contain $(2^t-1)$-hooks. 
Hence we obtain that $\chi^\mu(g)=0$. This is a contradiction, since again by the Murnaghan-Nakayama rule it is easy to check that $\chi^\mu(g)>0$. 

\smallskip

\textbf{(2)} Suppose that $j\in\{0,2^t-1\}$.
Then $\chi^\nu(1)=1$. Let $g\in \fS_n$ be a $2^t$-cycle such that $g(x)=x$ for all $x\in\{2^k+1,2^k+2,\ldots ,n\}$, and let $h\in \fS_n$ be the $2^t$-cycle defined $$h=(2^k+1,2^k+2,\ldots ,n).$$
Exactly as before we have that $$\chi^\mu(g)=\chi^\mu(g)\chi^\nu(1)=\chi^\mu(1)\chi^\nu(h)=\pm\chi^\mu(1).$$
By the Murnaghan-Nakayama rule we have that $|\chi^{\mu}(g)|\leq \chi^\rho(1)$, where $\rho$ is a hook partition of $2^k-2^t$ with leg length equal to either $i$ or $i-2^t$ (if this number is non-negative). In both cases we have that $$\chi^\mu(1)=|\chi^{\mu}(g)|\leq\chi^\rho(1)<\binom{2^k-1}{i}=\chi^\mu(1).$$ 
This is a contradiction. 
\end{proof}

\begin{lemma}\label{lem:reducS}
Let $n\in\mathbb{N}$. Suppose that there exists $t\geq 2$ and  $k_1>k_2>\cdots >k_t\geq 0$ such that $n=2^{k_1}+\cdots +2^{k_t}$ is the $2$-adic expansion of $n$. Let $\lambda\in \mathcal{P}(n)$ be such that $\chi^\lambda(1)$ is odd and $\chi^\lambda(1)>1$. 
Then $\chi^\lambda\res_{\fS_{2^{k_1}}\times\cdots\times \fS_{2^{k_t}}}$ is not irreducible. 
\end{lemma}

\begin{proof}
We proceed by induction on $t$. The base case $t=2$ is proved in Lemma \ref{lem:nightmare}. 
Suppose that $t>2$. Let $\rho\in \mathcal{P}(n)$ and $\tau\in\mathcal{P}(n-2^{k_t})$ be the partitions defined by defined by 
$$\rho=(2^{k_1}, 2^{k_2},\ldots, 2^{k_t})\ \ ,\ \ \tau=(2^{k_1}, 2^{k_2},\ldots, 2^{k_{t-1}}).$$
Denote by $\fS_\rho$ and $\fS_\tau$ the corresponding Young subgroups of $\fS_n$. 
Let $H=\fS_{n-2^{k_t}}\times \fS_{2^{k_t}}$ and $K=\fS_{2^{k_1}}\times \fS_{n-2^{k_1}}$. Clearly $\fS_\rho\leq H\cap K$ and $\chi^\lambda\res_{\fS_{\rho}}=(\chi^\lambda\res_H)\res_{\fS_{\rho}}$. 
If $\chi^\lambda\res_H$ is reducible, then we are done. 
Suppose that $\chi^\lambda\res_H=\chi^\mu\times\chi^\nu$ for some $\mu\in\mathcal{P}(n-2^{k_t})$ and $\nu\in\mathcal{P}(2^{k_t})$. 
Obviously $\chi^\mu(1)$ and $\chi^\nu(1)$ are both odd. 
If $\chi^\mu(1)\neq 1$ then $\chi^\mu\res_{\fS_{\tau}}$ is reducible by inductive hypothesis. Since $\fS_\rho=\fS_\tau\times \fS_{2^{k_t}}$, we deduce that $\chi^\lambda\res_{\fS_{\rho}}$ is reducible. 
On the other hand if $\chi^\mu(1)=1$ then $\chi^\nu(1)\neq 1$, because $\chi^\lambda(1)\neq 1$. Therefore, there exists $\zeta\in \mathcal{P}(n-2^{k_1})$ such that $\chi^\zeta(1)\neq 1$ and $\chi^\lambda\res_K=\phi\times \chi^\zeta,$
for some linear character $\phi$ of $\fS_{2^{k_1}}$. 
We can use again the inductive hypothesis to deduce that $\chi^\zeta\res_{\fS_{2^{k_2}}\times\cdots\times \fS_{2^{k_t}}}$ is reducible and therefore that also $\chi^\lambda\res_{\fS_{\rho}}$ is reducible. 
\end{proof}

We are now ready to prove Theorem \ref{thmB}.

\begin{proof}[Proof of Theorem \ref{thmB}]
Let $n=2^{k_1}+2^{k_2}+\cdots +2^{k_t}$ be the $2$-adic expansion of $n$, where $k_1>k_2>\cdots >k_t\geq 0$.
We proceed by induction on $t$.
If $t=1$ then $n$ is a power of $2$ and the first statement follows by Corollary \ref{cor:11}. 
Assume that $t\geq 2$. Then $$P_n=P_{2^{k_1}}\times\cdots\times P_{2^{k_t}}\leq \fS_{n-2^{k_t}}\times \fS_{2^{k_t}}=:H.$$  
We have that 
$$\chi^{\lambda}\res_H=\sum_{\mu, \nu}c_{\mu\nu}^\lambda(\chi^\mu\times\chi^\nu),$$
where the sum is taken over all the partitions $\mu\in\mathcal{P}(n-2^{k_t})$ and $\nu\in\mathcal{P}(2^{k_t})$. 
Let $\mu\in\mathcal{P}(n-2^{k_t})$ and $\nu\in\mathcal{P}(2^{k_t})$ be such that $c_{\mu\nu}^\lambda\neq 0$. Then by inductive hypothesis we have that $\chi^\mu\res_{P_{n-2^{k_t}}}$ 
has a linear constituent. The same holds for $\chi^\nu\res_{P_{2^{k_t}}}$. Since $P_n=P_{n-2^{k_t}}\times P_{2^{k_t}}$, we conclude that $\chi^\lambda\res_{P_n}$ has a linear constituent. This completes the proof of part (i).

To prove (ii), we first observe that if $n=2^k$ and $\chi^\lambda(1)$ is odd then by Theorem \ref{Thm:Alperin}, we have that $\chi^\lambda\res_{P_n}$ has a unique linear constituent. 
On the other hand, if $\chi^\lambda(1)$ is even then $\chi^\lambda\res_{P_n}$ has a non-zero, even number of linear constituents. 
Hence it remains to show that whenever $n$ has $2$-adic expansion given by $n=2^{k_1}+2^{k_2}+\cdots +2^{k_t}$, for some $t\geq 2$ and $k_1>k_2>\cdots >k_t\geq 0$, then for all $\lambda\in\mathcal{P}(n)$ such that $\chi^\lambda(1)>1$ is odd we have that $\chi^\lambda\res_{P_n}$ has more than one linear constituent.
Let $\rho$ be the partition defined by $$\rho=(2^{k_1},2^{k_2},\ldots ,2^{k_t}),$$
and let $\fS_\rho$ be the corresponding Young subgroup.
Clearly $P_n$ is a Sylow $2$-subgroup of $\fS_{\rho}$.
By Corollary \ref{cor:11} we deduce that the number of linear constituents of $\chi^\lambda\res_{P_n}$ is greater or equal than the number of irreducible constituents of $\chi^\lambda\res_{\fS_\rho}$.
Since $\lambda\notin\{(n), (1^n)\}$ we have that $\chi^\lambda\res_{\fS_\rho}$ is not irreducible, by Lemma \ref{lem:reducS}.
\end{proof}

\medskip




\subsection*{Acknowledgements}
This paper grew out of conversations with Gabriel Navarro while I was
visiting the University of Valencia. I am grateful to him for his fundamental advice throughout this work.  
I thank all the members of the department of mathematics of the University of Valencia, especially Carolina Vallejo for many interesting conversations on the McKay conjecture.
Also, my thanks to Paul Fong for valuable suggestions on this note. 
Finally, I am grateful to Gunter Malle for many helpful comments on a previous version of this article.

\end{document}